\documentclass[authoryear]{article}
\usepackage{amsmath,amsthm,amsopn,amstext,amscd,amsfonts,amssymb,verbatim}
\usepackage{color}
\usepackage[numbers]{natbib}

\usepackage{graphicx}
\usepackage{verbatim,color}
\usepackage{animate}

\def\tr{\mathop{\rm tr}\nolimits}

\def\etr{\mathop{\rm etr}\nolimits}

\newcommand {\boldgreektext}[1] {\boldmath
             \(#1\)\unboldmath}
\newcommand {\boldgreek}[1]
             {\mbox{\boldgreektext{#1}}
            }

\renewenvironment{abstract}
                 {\vspace{6pt}
                  \begin{center}
                  \begin{minipage}{5in}
                  \centerline{\textbf{Abstract}}
                  \noindent\ignorespaces
                 }
                 {\end{minipage}\end{center}}
                 
\newtheorem{theorem}{\textbf{Theorem}}[section]

\theoremstyle{definition}

\title{\Large \textbf{Matrix generalized elliptical binomial series under real normed division algebras and the central matrix variate beta distribution}}
\author{
  \textbf{Francisco J. Caro-Lopera} \thanks{Corresponding author\newline
    }\\
  {\normalsize University of Medellin} \\
  {\normalsize Faculty of Basic Sciences} \\
  {\normalsize Carrera 87 No.30-65} \\
  {\normalsize Medell\'{\i}n, Colombia} \\
  {\normalsize E-mail: fjcaro@udemedellin.edu.co} \\
  \textbf{Jos\'e A. D\'{\i}az-Garc\'{\i}a}\\
  {\normalsize Universidad Aut\'onoma de Chihuahua} \\
  {\normalsize Facultad de Zootecnia y Ecolog\'{\i}a} \\
  {\normalsize Perif\'erico Francisco R. Almada Km 1, Zootecnia} \\
  {\normalsize 33820 Chihuahua, Chihuahua, M\'exico}\\
  {\normalsize E-mail: jadiaz@cimat.mx}\\[2ex]
}
\date{}
\begin{document}
\maketitle

\begin{abstract}
In this paper we provide a matrix extension of the scalar binomial series under elliptical contoured models and real normed division algebras. The classical hypergeometric series ${}_{1}F_{0}^{\beta}(a;\mathbf{Z})={}_{1}^{k}P_{0}^{\beta,1}(1:a;\mathbf{Z})=|\mathbf{I}-\mathbf{Z}|^{-a}$ of Jack polynomials are now seen as an invariant generalized determinant with a series representation indexed by any elliptical generator function. In particular, a corollary emerges for a simple derivation of the central matrix variate beta type II distribution under elliptically contoured models in the unified real, complex, quaternions and octonions.  
\end{abstract}

\section{Introduction}\label{sec:1}

A unified matrix variate distribution theory has recently emerged into a simple dimension-parameterized description of the independently developed classical statistics for real ($\beta=1$), complex ($\beta=2$), quaternion ($\beta=4$) and octonion ($\beta=8$); see for example \citet{dgcl:16} and the references therein. The usual treatment of separated Gaussian, Pearson, Kotz type matrix variate models have been also unified into the real elliptically contoured distributions in functions of traces and other generalizations by \citet{gv:93} and \citet{fz:90}, respectively. Melting both unifications provide results that enhances the matrix theory in a more robust general setting. Moreover, the addressed elliptical case also can expect some new approaches. For example, by using a connection between the determinant and the permanent, \citet{clgfbn:13} opens a perspective into the so termed generalized distributions under permanents. However several new results are needed for determinants in order to be translated for permanents. In particular, \citet{clgfbn:13} motivate the generalized series of determinants of the present work for the following real case $\beta=1$. Let $(d\mathbf{H})$ be the invariant normalized Haar measure on the orthogonal group of $m\times m$ matrices. Let also $\etr(\cdot)\equiv\exp(\tr(\cdot))$,  $Re(a)>(m-1)/2$. If $\mathbf{X}_{1}$ and $\mathbf{X}_{2}$ are positive definite $m\times m$ matrices, then we expect a solution for the claimed parallelism of a non central determinantal elliptical distributions proposed by \citet{clgfbn:13} via a possible new polynomial expansion base for $\int_{O(m)}|\mathbf{I}+\mathbf{X}_{1}\mathbf{H}\mathbf{X}_{2}\mathbf{H}'|^{-a}(d\mathbf{H})$. A fact that avoids the current expansion into invariant polynomials of \citet{d:79} and the implicit impossibility of computation due to \citet{cl:16}. A feasible procedure that can emulate the fundamental averaging $\int_{O(m)}\etr(\mathbf{X}_{1}\mathbf{H}\mathbf{X}_{2}\mathbf{H}')(d\mathbf{H})$ of \citet{j:60} for generating the zonal polynomials.  A first step for that task requires a characterization of ${}_{1}F_{0}^{\beta}(a;\mathbf{Z})={}_{1}^{k}P_{0}^{\beta,1}(1:a;\mathbf{Z})=|\mathbf{I}-\mathbf{Z}|^{-a}$ into the general setting of real normed division algebras (RNDA) and classical elliptical models based on traces by \citet{gv:93}. In this context, Section \ref{sec:2} provides the required series representation of $|\mathbf{I}-\mathbf{Z}|^{-a}$ which is invariant under the complete family of elliptically contoured matrix variate distributions and works for every RNDA. Finally, as a consequence, Section \ref{sec:3} gives a new derivation of the central beta distribution under RNDA, which is invariant under elliptical models and non isotropic covariance matrices, as we expect.

\section{Matrix generalized elliptical binomial series under real normed division algebras}\label{sec:2}

For the sequel, complete details can be seen in \citet{dgcl:16} and \citet{dg:14}. Let $\mathfrak{F}$ be a RNDA, where $\mathfrak{F}^{m\times n}$ denotes the set of all $n\times m$ matrices over $\mathfrak{F}$ and its real dimension is $\beta m n$.  Define $\mathfrak{S}_{m}^{\beta}$ the real vector space of all $\mathbf{Z}\in\mathfrak{F}^{m\times m}$, such that $\mathbf{Z}=\mathbf{Z}^{*}$, where $\mathbf{Z}^{*}$ is the conjugate transpose and $||\mathbf{Z}||$ denotes the absolute value of its maximum eigenvalue. $\mathfrak{P}_{m}^{\beta}$ is an open subset of $\mathfrak{S}_{m}^{\beta}$, consisting of $\mathbf{W}=\mathbf{X}^{*}\mathbf{X}$, with $\mathbf{X}\in\mathfrak{L}_{m,n}^{\beta}$, the set of all $m\times n$ of rank $m\leq n$ over $\mathfrak{F}$ with $m$ distinct positive singular values. $(d\mathbf{X})$ is the Lebesgue measure of $\mathfrak{F}^{m\times n}$, denoting the exterior product of the corresponding $\beta m n$ functionally independent variables. $\mathfrak{U}^{\beta}(m)$ represents the groups: real orthogonal $O(m)$, unitary, compact symplectic and exceptional type matrices, for respectively $\beta$; meanwhile, $(\mathbf{H}^{*}d\mathbf{H})$ is the Haar measure on    $\mathfrak{U}^{\beta}(m)$. 

Now, $\mathbf{X}\in\mathfrak{L}_{m,n}^{\beta}$ has a matrix variate elliptically contoured distribution under RNDA (denoted by $\mathbf{X}\sim \mathcal{E}_{n\times m}^{\beta}\left(\boldsymbol{\mu}_{n\times m},\boldsymbol{\Theta}_{n\times n}\otimes\boldsymbol{\Sigma}_{m\times m},h\right)$) if its density respect to the Lebesgue measure is given by:
$$
  f_{\mathbf{X}}(\mathbf{X})=\frac{1}{|\boldsymbol{\Theta}|^{\beta m/2}|\boldsymbol{\Sigma}|^{\beta n/2}}h\left\{\tr[\boldsymbol{\Sigma}^{-1}(\mathbf{X}-\boldsymbol{\mu})^{*}\boldsymbol{\Theta}^{-1}(\mathbf{X}-\boldsymbol{\mu})]\right\},
$$
where $\boldsymbol{\mu}\in \mathfrak{L}_{m,n}^{\beta}$, $\boldsymbol{\Theta}\in\mathfrak{B}_{n}^{\beta}$, $\boldsymbol{\Sigma}\in\mathfrak{B}_{m}^{\beta}$ and the generator function  $h:\mathfrak{F}\rightarrow [0,\infty )$ satisfies $\int_{u\in\mathfrak{P}_{1}^{\beta}}h(u^{2})u^{\beta m n-1}du<\infty$.

Finally, generalized hypergeometric series of Jack polynomials $C_{\kappa}^{\beta}(\mathbf{U})$, indexed by positive ordered integer partitions $\kappa=(k_{1},\ldots,k_{m})$ of $k$, are defined by: 
$$
  {}_{p}^{k}P_{q}^{\beta,1}\left(g(\mathbf{U}):a_{1},\ldots,a_{p};b_{1},\ldots,b_{q};\mathbf{U}\right)= \sum_{k=0}^{\infty} \frac{g(\mathbf{U})}{k!} \sum_{\kappa}\frac{(a_{1})_{\kappa}^{\beta}\cdots (a_{p})_{\kappa}^{\beta}}{(b_{1})_{\kappa}^{\beta}\cdots (b_{q})_{\kappa}^{\beta}} C_{\kappa}^{\beta}(\mathbf{U})
$$
and 
$$
  {}_{p}^{k}P_{q}^{\beta,2}\left(r(\mathbf{U},\mathbf{V}):a_{1},\ldots,a_{p};b_{1},\ldots,b_{q};\mathbf{U};\mathbf{V}\right)=\hspace{4.5cm}
$$
$$\hspace{4cm}
   \sum_{k=0}^{\infty}\frac{g(\mathbf{U})}{k!}\sum_{\kappa}\frac{(a_{1})_{\kappa}^{\beta}\cdots (a_{p})_{\kappa}^{\beta}}{(b_{1})_{\kappa}^{\beta}\cdots (b_{q})_{\kappa}^{\beta}}\frac{C_{\kappa}^{\beta}(\mathbf{U})C_{\kappa}^{\beta}(\mathbf{V})}{C_{\kappa}^{\beta}(\mathbf{I})},
$$
where $g(\cdot), r(\cdot, \cdot)$ are functions of $\mathbf{U}, \mathbf{V}\in\mathfrak{S}_{m}^{\beta}$, and $(a_{i})_{\kappa}^{\beta}$, $i=1, \dots,p$ and $(b_{j})_{\kappa}^{\beta}$, $j = 1, \dots, q$ are the generalized Pochhammer coefficients (see \citet{dg:14} and \citet{cldg:12}).

Now, the real hypergeometric series ${}_{1}^{k}P_{0}^{1,1}(1:a;\mathbf{Z})=|\mathbf{I}-\mathbf{Z}|^{-a}$ of zonal polynomials is an old result that generalises the classical
scalar binomial series. However, its derivation depends strongly of the Gaussian kernel, implicitly present in a matrix Laplace transform, see \citet[Corollary 7.3.5.]{mh:05}. The next theorem establishes that the so termed matrix generalized elliptical binomial series is invariant under the family of elliptically contoured distributions over the RNDA.  
\begin{theorem}\label{Th:EllipticalDeterminant}
    Let $Re(a)>(m-1)\beta/2$ and $\mathbf{Z}\in\mathfrak{S}_{m}^{\beta}$, with $||\mathbf{Z}||<1$. Then,
\begin{equation}\label{eq:EllipticalDeterminant}
    |\mathbf{I}-\mathbf{Z}|^{-a}=\pi^{ma}\,{}_{1}^{k}P_{0}^{\beta,1}\left(\displaystyle\frac{\int_{w\in\mathfrak{P}_{1}^{\beta}}
    h^{(k)}(w)w^{ma+k-1}dw}{\Gamma_{1}^{\beta}(ma+k)}:a;\mathbf{-Z}\right),
\end{equation}
where the generator $h:\mathfrak{F}\rightarrow [0,\infty )$ satisfies $\int_{y\in\mathfrak{P}_{1}^{\beta}}h(y)y^{am+k-1}dy<\infty$.
\end{theorem}
\begin{proof}
Take $\kappa=0$ in \citet[Lemma 2]{dgcl:16}, then $|\mathbf{I}-\mathbf{Z}|^{-a}$ can be written as
$$ 
  \frac{\Gamma_{1}^{\beta}(ma)}{\Gamma_{m}^{\beta}(a)\int_{y\in\mathfrak{P}_{1}^{\beta}}h(y)y^{ ma-1}dy}\int_{\mathbf{X}\in\mathfrak{P}_{m}^{\beta}}h\left(\tr \mathbf{X}-\tr\mathbf{X}\mathbf{Z}) \right)|\mathbf{X}|^{a-(m-1)\beta/2-1}(d\mathbf{X}).
$$
After expanding in a convergent series of Jack polynomials around $\tr\mathbf{X}$ we have
$$ 
  \frac{\Gamma_{1}^{\beta}(ma)\left[\Gamma_{m}^{\beta}(a)\right]^{-1}}{\int_{y\in\mathfrak{P}_{1}^{\beta}} h(y)y^{ma-1}dy}\int_{\mathbf{X} \in \mathfrak{P}_{m}^{\beta}} |\mathbf{X}|^{a-(m-1)\beta/2-1}\,{}_{0}^{k}P_{0}^{\beta,1}\left(h^{(k)}\left(\tr \mathbf{X} \right)
  :-\mathbf{Z}\mathbf{X}\right)(d\mathbf{X}).
$$
Integration by using \citet[Lemma 2]{dgcl:16} leads to:
$$ 
  \frac{\Gamma_{1}^{\beta}(ma)\left[\Gamma_{m}^{\beta}(a)\right]^{-1}}{\int_{y\in\mathfrak{P}_{1}^{\beta}}h(y)y^{ma-1}dy}\,{}_{1}^{k} P_{0}^{\beta,1}\left(\displaystyle\frac{\int_{w\in\mathfrak{P}_{1}^{\beta}}
   h^{(k)}(w)w^{ma+k-1}dw} {\Gamma_{1}^{\beta}(ma+k)}:a;-\mathbf{Z}\right).
$$
Finally, extending the real case of \citet[p. 59]{fz:90} into RNDA gives the required result. 
\end{proof}         
 
\section{Matrix variate beta type I and II distributions based on elliptically contoured models for real normed division algebras}\label{sec:3}

In this section we provide an application of Theorem \ref{Th:EllipticalDeterminant} in the setting of the central matrix variate beta type II distribution.

Real ($\beta=1$) central matrix variate beta type I and II distributions based on a matrix variate Gaussian distribution back to \citet{h:39} and they are studied in \citet{mh:05}, among many others. The invariance of the central matrix variate beta type I and II distributions under elliptical contoured models in terms of traces arrived later in \citet[p. 182]{gv:93} and the references therein.

We now study the central matrix variate beta distribution under the unified approach of RNDA and elliptically contoured models. As we shall see the solution is just a simple consequence of Theorem \ref{Th:EllipticalDeterminant}, and our derivation revisits some historical problems in the Gaussian case.

Let $\mathbf{X}=(\mathbf{X}_{1}^{*},\mathbf{X}_{2}^{*})^{*}\sim \mathcal{E}_{(n_{1}+n_{2})\times m}^{\beta}\left(\mathbf{0}_{(n_{1}+n_{2})\times m},\mathbf{I}_{n_{1}+n_{2}}\otimes\boldsymbol{\Sigma}_{m\times m},h\right)$. Then the joint distribution of $\mathbf{X}_{1},\mathbf{X}_{2}$ is given by:
$$
  dF_{\mathbf{X}}(\mathbf{X})=|\boldsymbol{\Sigma}|^{-\beta(n_{1}+n_{2})/2}h\left[\tr\boldsymbol{\Sigma}^{-1}(\mathbf{X}_{1}^{*} \mathbf{X}_{1}+\mathbf{X}_{2}^{*}\mathbf{X}_{2})\right](d\mathbf{X}_{1})\wedge(d\mathbf{X}_{2}).
$$
Define $\mathbf{W}_{i}=\mathbf{X}_{i}^{*}\mathbf{X}_{i}, i=1,2$. By \citet{dgcl:24b}: 
$$
 (d\mathbf{X}_{i})=2^{-m}|\mathbf{W}_{i}|^{\beta(n_{i}-m+1)/2-1}(d\mathbf{W}_{i})(\mathbf{H}_{i}^{*}d\mathbf{H}_{i}), \quad i=1,2. 
$$
Then integration over $\mathfrak{U}^{\beta}(m)$, by using \citet{dgcl:16},  provides the joint distribution of generalized elliptical Wishart matrices $\mathbf{W}_{1}$, $\mathbf{W}_{2}$:
$$
  dF_{\mathbf{W}_{1},\mathbf{W}_{2}}(\mathbf{W}_{1},\mathbf{W}_{2})=\frac{\pi^{\beta(n_{1}+n_{2})m/2}|\boldsymbol{\Sigma}|^{-\beta(n_{1}+n_{2})/2}} {\Gamma_{m}^{\beta}\left(\frac{\beta n_{1}}{2}\right)\Gamma_{m}^{\beta }\left(\frac{\beta n_{2}}{2}\right)}|\mathbf{W}_{1}|^{\beta (n_{1}-m+1)/2-1}\hspace{4cm}$$
$$
  \hspace{2cm}\times |\mathbf{W}_{2}|^{\beta (n_{2}-m+1)/2-1}h\left[tr\boldsymbol{\Sigma}^{-1}(\mathbf{W}_{1}+\mathbf{W}_{2})\right](d\mathbf{W}_{1})\wedge(d\mathbf{W}_{2}),
$$
a fact denoted by  $(\mathbf{W}_{1},\mathbf{W}_{2})^{*} \sim \mathcal{EW}_{m}^{\beta}(n_{1}, n_{2},\boldsymbol{\Sigma}, \mathbf{\Sigma};h)$.

Now, consider the following transformations: $\mathbf{F}=\mathbf{W}_{2}^{-1/2}\mathbf{W}_{1}\mathbf{W}_{2}^{-1/2}$ and
$$
  \mathbf{W}=\begin{pmatrix}
                \mathbf{W}_{1}\\
                \mathbf{W}_{2}
             \end{pmatrix}
             =\mathbf{W}_{2}^{1/2}
             \begin{pmatrix}
                \mathbf{W}_{2}^{-1/2}\mathbf{W}_{1}\mathbf{W}_{2}^{-1/2}\\
                \mathbf{I}_{m}\\
             \end{pmatrix}\mathbf{W}_{2}^{1/2} 
             =\begin{pmatrix}
                 \mathbf{W}_{2}^{1/2}\mathbf{F}\mathbf{W}_{2}^{1/2}\\
                 \mathbf{W}_{2}^{1/2}\mathbf{W}_{2}^{1/2}
              \end{pmatrix},
$$ then $(d\mathbf{W})=|\mathbf{W}_{2}|^{\beta(m-1)/2+1}(d\mathbf{F})\wedge(d\mathbf{W}_{2})$.

The joint density function of $\mathbf{F}$ and $\mathbf{W}_{2}$ is reduced to:
$$
   dF_{\mathbf{W}_{2},\mathbf{F}}(\mathbf{W}_{2},\mathbf{F})=\frac{\pi^{\beta(n_{1}+n_{2})m/2}|\boldsymbol{\Sigma}|^{-\beta(n_{1}+n_{2})/2}} {\Gamma_{m}^{\beta }\left(\frac{\beta n_{1}}{2}\right)\Gamma_{m}^{\beta }\left(\frac{\beta n_{2}}{2}\right)} |\mathbf{W}_{2}|^{\beta(n_{1}+n_{2}-m+1)/2-1}\hspace{4cm}
$$
\begin{equation}\label{integral}
    \hspace{1cm}\times |\mathbf{F}|^{\beta(n_{1}-m+1)/2-1}h\left[\tr\boldsymbol{\Sigma}^{-1}\mathbf{W}_{2}^{1/2}(\mathbf{I}+\mathbf{F}) \mathbf{W}_{2}^{1/2}\right](d\mathbf{W}_{2})\wedge(d\mathbf{F}).
\end{equation}
Thus, the distribution of $\mathbf{F}$ follows by integration of $\mathbf{W}_{2}$ over $\mathfrak{P}_{m}^{\beta}$.

Certain real Gaussian kernel ($h(y)=\exp(-y/2)$) of the required integral appeared in the context of noncentral beta type I and it was declared by \citet{c:63} in the following terms: 

\medskip
\begin{small}
\textbf{\textit{``It would be of interest to be able to integrate out $\mathbf{G}$ ($\mathbf{W}_{2}$) from this expression and so obtain the distribution of the matrix $\mathbf{R}$ ($\mathbf{F}$), but this appears difficult unless $\boldsymbol{\Omega}$ ($\boldsymbol{\Sigma}$) is either a scalar matrix or of rank 1".}} 
\end{small}

\medskip
The problem preserved that label and was claimed in a number of papers and books like \citet{gn:00} and \citet{f:85}. The integral remained unsolved by 45 years until the old termed symmetrized distribution of \citet{g:73} was elucidated by   \citet{dggj:07} in the inverse way, for obtaining the corresponding nonsymmetrized  distribution.     

Now, a first inspection of the required integration in (\ref{integral}) is certainly more difficult than the addressed exponential core in \citet{c:63}, because, we expect that the generator must vanish into an invariant central beta independent of $h(\cdot)$ under the real normed division algebras.

The methodology of \citet{dggj:07} starts by obtaining the symmetrized distribution  $f_{s}(\mathbf{F})=\int_{\mathfrak{U}^{\beta}(m)}f(\mathbf{H}\mathbf{F}\mathbf{H}^{*})(d\mathbf{H})$. In this case, we require $$
  \int_{\mathfrak{U}^{\beta}(m)}h\left[\tr\boldsymbol{\Sigma}^{-1}\mathbf{W}_{2}^{1/2}\mathbf{H}(\mathbf{I}+\mathbf{F}) \mathbf{H}^{*}\mathbf{W}_{2}^{1/2}\right](d\mathbf{H}).
$$ 
Expanding in convergent series ${}_{0}^{k}P_{0}^{\beta,1}(\cdot)$ of Jack polynomials, indexed by the $k-$th derivative of the kernel elliptical model, we have 
$$
  \int_{\mathfrak{U}^{\beta}(m)}{}_{0}^{k}P_{0}^{\beta,1}\left(h^{(k)}(\tr\boldsymbol{\Sigma}^{-1}\mathbf{W}_{2}): \mathbf{W}_{2}^{1/2}\boldsymbol{\Sigma}^{-1}\mathbf{W}_{2}^{1/2}\mathbf{H}\mathbf{F}\mathbf{H}^{*}\right)(d\mathbf{H})
$$
and the integration via \citet{dgcl:16} leads to 
$$
  {}_{0}^{k}P_{0}^{\beta,2}\left(h^{(k)}(\tr\boldsymbol{\Sigma}^{-1}\mathbf{W}_{2}):\boldsymbol{\Sigma}^{-1}\mathbf{W}_{2};\mathbf{F}\right).
$$ 
Then the symmetrized function of (\ref{integral}) is simplified as
$$
  f_{s}(\mathbf{F})=\frac{\pi^{\beta(n_{1}+n_{2})m/2}|\boldsymbol{\Sigma}|^{-\beta(n_{1}+n_{2})/2}}{\Gamma_{m}^{\beta }\left(\frac{\beta n_{1}}{2}\right)\Gamma_{m}^{\beta }\left(\frac{\beta n_{2}}{2}\right)}|\mathbf{F}|^{\beta(n_{1}-m+1)/2-1}
\hspace{5cm}
$$
$$\hspace{1cm}
 \times\int_{\mathbf{W}_{2}\in\mathfrak{B}_{m}^{\beta}} |\mathbf{W}_{2}|^{\beta(n_{1}+n_{2}-m+1)/2-1}\quad{}_{0}^{k}P_{0}^{\beta,2} \left(h^{(k)}(\tr\boldsymbol{\Sigma}^{-1} \mathbf{W}_{2}):\boldsymbol{\Sigma}^{-1}\mathbf{W}_{2};\mathbf{F}\right)(d\mathbf{W}_{2}).    
$$
Thus, integration by using \citet{dgcl:16}  gets the form 
$$
f_{s}(\mathbf{F})=\frac{\pi^{\beta (n_{1}+n_{2})m/2}\Gamma_{m}^{\beta}\left(\frac{\beta(n_{1}+n_{2})}{2}\right)}{\Gamma_{m}^{\beta }\left(\frac{\beta n_{1}}{2}\right)\Gamma_{m}^{\beta }\left(\frac{\beta n_{2}}{2}\right)}|\mathbf{F}|^{\beta(n_{1}-m+1)/2-1}
\hspace{4cm}$$
\begin{equation*}\hspace{2cm}
\times{}_{1}^{k}P_{0}^{\beta,1}\left(\frac{\int_{w\in\mathfrak{P}_{1}^{\beta}} h^{(k)}(w) w^{\beta m(n_{1}+n_{2})/2+k-1}dw}{\Gamma_{1}^{\beta}\left(\frac{\beta m(n_{1}+n_{2})}{2}+k\right)}:\frac{\beta(n_{1}+n_{2})}{2};\mathbf{F}\right).    
\end{equation*}
Finally, by Theorem \ref{Th:EllipticalDeterminant}, the symmetrized function takes the simple distribution:
$$
 f_{s}(\mathbf{F})=\frac{\Gamma_{m}^{\beta}\left(\frac{\beta(n_{1}+n_{2})}{2}\right)}{\Gamma_{m}^{\beta}\left(\frac{\beta n_{1}}{2}\right)\Gamma_{m}^{\beta}\left(\frac{\beta n_{2}}{2}\right)}|\mathbf{F}|^{\beta(n_{1}-m+1)/2-1}|\mathbf{I}+\mathbf{F}|^{-\beta(n_{1}+n_{2})/2}.
$$
We end the procedure of \citet{dggj:07} in order to obtain the nonsymmetrized distribution function of $\mathbf{F}$. In this case we ask for a function $f(\mathbf{F})$ such that $\int_{\mathfrak{U}^{\beta}(m)}f(\mathbf{H}\mathbf{F}\mathbf{H}^{*})(d\mathbf{H})=f_{s}(\mathbf{F})$. And the solution easily establishes that $f(\mathbf{F})=f_{s}(\mathbf{F})$. 

Now, taking  $\mathbf{F}=\mathbf{U}(\mathbf{I}-\mathbf{U})^{-1}$, we have arrived at the matrix variate beta type I distribution under RNDA, which is also invariant under the elliptical contoured distributions.
\begin{theorem}\label{Th:beta}
Let the generalized elliptical Wishart matrices $\mathbf{W}_{1}$ and  $\mathbf{W}_{2}$ such that 
$$
 (\mathbf{W}_{1}, \mathbf{W}_{2})^{*} \sim \mathcal{EW}_{m}^{\beta}(n_{1},n_{2},\boldsymbol{\Sigma},\mathbf{\Sigma};h)
$$ 
and define $\mathbf{F}=\mathbf{W}_{2}^{-1/2}\mathbf{W}_{1}\mathbf{W}_{2}^{-1/2}$.  The distribution of the matrix variate beta type I, $\mathbf{U}=(\mathbf{I}+\mathbf{F}^{-1})^{-1}$, under any underlying contoured elliptical distribution in RNDA, is given by
\begin{equation}
  f_{\mathbf{U}}(\mathbf{U}) =\frac{\Gamma_{m}^{\beta}\left(\frac{\beta(n_{1}+n_{2})}{2}\right)}{\Gamma_{m}^{\beta}\left(\frac{\beta n_{1}}{2}\right)
  \Gamma_{m}^{\beta}\left(\frac{\beta n_{2}}{2}\right)}|\mathbf{U}|^{\beta(n_{1}-m+1)/2-1}|\mathbf{I}-\mathbf{U}|^{\beta(n_{2}-m+1)/2-1},
\end{equation}
where $n_{1}>\beta(m-1)/2,n_{2}>\beta(m-1)/2$ and $\mathbf{0}<\mathbf{U}<\mathbf{I}$.
\end{theorem}
Note that Theorem \ref{Th:EllipticalDeterminant} is an independent concept of beta distribution. It is just related with a new integral representation of the hypergeometric series ${}_{1}F_{0}^{\beta}(a;\mathbf{Z})={}_{1}^{k}P_{0}^{\beta,1}(1:a;\mathbf{Z})$, under RNDA and elliptically contoured kernel functions.  About the beta distribution derivation, the use of symmetrized and nonsymmetrized functions is an available technique which can be tested in other matrix factorization for filtering geometric information in quotient spaces (see \citet{dgcl:16}).  
Finally, the provided technique can be used for obtaining the distributions of rectangular or square nonsymmetric matrices. For example, instead of the matrix $\mathbf{F}=\mathbf{W}_{2}^{-1/2}\mathbf{W}_{1}\mathbf{W}_{2}^{-1/2}$ we can define the nonsymmetric matrix $\mathbf{G}=\mathbf{W}_{1}\mathbf{W}_{2}^{-1}$, which the latent roots are the same than $\mathbf{F}$. Thus, once the  Jacobian $(d\mathbf{G})=J(\mathbf{W}_{1} \rightarrow \mathbf{G})(d\mathbf{W}_{2})$ is reached, then the distribution of  $\mathbf{G}$ follows by a similar procedure of Theorem \ref{Th:beta} applied to 
$$
  \mathbf{W}=\begin{pmatrix}
                \mathbf{W}_{1}\\
                \mathbf{W}_{2}
              \end{pmatrix}
              =\begin{pmatrix}
                 \mathbf{W}_{1}\mathbf{W}_{2}^{-1}\\
                 \mathbf{I}_{m}\\
               \end{pmatrix}
               \mathbf{W}_{2}
               =\begin{pmatrix}
                 \mathbf{G}\mathbf{W}_{2}\\
                 \mathbf{W}_{2}\\
               \end{pmatrix}.
$$

\section{Conclusions}
This work has presented the binomial series in the setting of two generalized unifications: the matrix variate elliptically contoured distributions and the RNDA. It is a first step for a future averanging $\int_{\mathfrak{U}^{\beta}(m)}|\mathbf{I}+\mathbf{X}_{1}\mathbf{H}\mathbf{X}_{2}\mathbf{H}'|^{-a}(d\mathbf{H})$ under RNDA in order to avoid non computable Davis polynomials and explore the noncentral matrix variate symmetric distributions in terms of permanents and determinants. The generalized series here derived has led to a simple construction of the central matrix variate beta type II distribution by a trivial equality of symmetrized and nonsymmetrized functions. A future application of the matrix elliptical series under RNDA can be seen in distribution theory of permanents via \citet{clgfbn:13}. 




\begin{thebibliography}{}

    \bibitem[Caro-Lopera(2016)]{cl:16}
    Caro-Lopera, F. J., 2016.
    The impossibility of a recurrence construction of the invariant polynomials by using the Laplace-Beltrami operator. 
    Far East J.Math. Sci. 100(8), 1265--1288.

    \bibitem[Caro-Lopera and D\'{\i}az-Garc\'{\i}a (2012)]{cldg:12}
    Caro-Lopera, F. J., D\'{\i}az-Garc\'{\i}a, J.A., 2012.
    Matrix Kummer-Pearson VII Relation and Polynomial Pearson VII Configuration Density.
    J. Iran. Stat. Soc., 11, 217--230.

     \bibitem[Caro-Lopera \textit{et al.}(2013)]{clgfbn:13}
    Caro-Lopera, F. J., Gonz\'alez-Far\'{\i}as, G., Balakrishnan, N., 2013. 
    Determinants, permanents and some applications to statistical shape theory. 
    J. Multivariate Anal. 114, 29--39.
    
    \bibitem[Constantine (1963)]{c:63}
    Constantine, A. G., 1963.
    Some non-central distribution problems in multivariate analysis. 
    Ann. Math. Statist. 34, 1270--1285.

\bibitem[Davis (1979)]{d:79}
    Davis, A. W., 1979. 
    Invariant polynomials with two matrix arguments extending the zonal polynomials: 
    Applications to multivariate distribution theory. 
    Ann. Inst. Stat. Math. 31, 465--485.

\bibitem[D\'{\i}az-Garc\'{\i}a (2014)]{dg:14}
    D\'{\i}az-Garc\'{\i}a, J. A., 2014.
    Integral Properties of Zonal Spherical Functions, Hypergeometric Functions and Invariant Polynomials.
    J. Iran. Stat. Soc., 13(1), 83-124.

\bibitem[D\'{\i}az-Garc\'{\i}a and Guti\'{e}rrez-J\'{a}imez(2007)]{dggj:07}
   D\'{\i}az-Garc\'{\i}a J. A., Guti\'{e}rrez-J\'{a}imez, R., 2007. 
   Noncentral, nonsingular matrix variate beta distribution. 
   Brazilian J. Prob. Statist 21, 175–186.

\bibitem[D\'{\i}az-Garc\'{\i}a and Caro-Lopera(2016)]{dgcl:16}
    D\'{\i}az-Garc\'{\i}a, J. A., Caro-Lopera, F. J., 2016.
    Elliptical affine shape distributions for real normed
    division algebras.
    J. Multivariate. Anal.  144, 139-149.

 \bibitem[D\'{\i}az-Garc\'{\i}a and Caro-Lopera(2024)]{dgcl:24b}
    D\'{\i}az-Garc\'{\i}a, J. A., Caro-Lopera, F. J., 2024.
   Multimatricvariate and multimatrix variate distributions based on elliptically contoured laws under real normed division algebras. 
   arXiv:2405.06905. Also submitted. 

\bibitem[Fang and Zhang (1990)]{fz:90}
    Fang, K. T., Zhang, Y. T., 1990.
    Generalized Multivariate Analysis.
    Science Press, Springer-Verlag, Beijing.

    \bibitem[Farrell (1985)]{f:85}
    Farrell, R. H., 1985. 
    Multivariate Calculation: Use of the Continuous Groups. 
    Springer Series in Statistics, Springer-Verlag, New York.


\bibitem[Greenacre (1973)]{g:73}
   Greenacre, M. J., 1973. 
   Symmetrized multivariate distributions. 
   S. Afr. Statist. J. 7,95–101.
   
\bibitem[Gupta and Nagar(2000)]{gn:00} 
    Gupta, A. K.,  Nagar, D. K., 2000. 
    Matrix variate distributions.  
    Chapman \& Hall/CR, New York.
          
    \bibitem[Gupta and Varga (1993)]{gv:93}
    Gupta, A. K., Varga, T., 1993.
    Elliptically Contoured Models in Statistics,
    Kluwer Academic Publishers, Dordrecht.

\bibitem[Hsu(1939)]{h:39} Hsu. P. L., 1939. 
    On the distribution of the roots of certain determinantal equations. 
    Ann. Eugen., 9, 250-258.

\bibitem[James(1960)]{j:60}
   James, A. T., 1960. 
   The Distribution of the latent roots of the covariance matrix. 
   Ann. Math. Statist. 31(1), 151-158.
    
   \bibitem[Muirhead(2005)]{mh:05}
   Muirhead, R. J., 2005.
   Aspects of Multivariate Statistical Theory.
   John Wiley \& Sons, New York.

 \end{thebibliography}
\end{document}